\documentclass[12pt, reqno]{amsart}

\usepackage{textcomp} 
\usepackage{amsmath, amsfonts,amsthm,amssymb,amsxtra}
\usepackage{amscd}
\usepackage{enumitem}

\usepackage{xcolor}
\usepackage{graphicx}
\usepackage{verbatim}

\newtheorem{theorem}{Theorem}[section]

\newtheorem{corollary}[theorem]{Corollary}
\newtheorem{question}[theorem]{Question}

\theoremstyle{definition}

\theoremstyle{remark}

\newtheorem{remark}[theorem]{Remark}

\numberwithin{equation}{section}

\renewcommand{\epsilon}{\varepsilon}

\renewcommand{\phi}{\varphi}

\begin{document}

{
\title{H-principles for regular Lagrangians}
}
\author{Oleg Lazarev }
\thanks{The author was supported by an NSF postdoc fellowship. }
\address{Oleg Lazarev, Columbia University}
\email{olazarev@math.columbia.edu}  

\maketitle

\begin{abstract}
We prove an existence h-principle for regular Lagrangians with Legendrian boundary in arbitrary Weinstein domains of dimension at least six; this extends a previous result of Eliashberg, Ganatra, and the author for Lagrangians in flexible domains. Furthermore, we show that all regular Lagrangians come from our construction and  describe some related decomposition results. We also prove a regular version of Eliashberg and Murphy's h-principle for Lagrangian caps with loose negative end. As an application, 
we give a new construction of infinitely many regular Lagrangian disks in the standard Weinstein ball. 
\end{abstract}

\section{Introduction}

\subsection{Existence h-principle for Lagrangians}

In \cite{EGL}, Eliashberg, Ganatra, and the author introduced the class of \textit{regular} Lagrangians in Weinstein domains. These Lagrangians have Weinstein complement and hence can be constructed by coupled Weinstein handle attachment. Regular Lagrangians have the advantage that they can be manipulated via Weinstein homotopy moves and studied via Legendrian knot theory. The special class of \textit{flexible} Lagrangians, which have flexible Weinstein complement, was also defined in \cite{EGL}. It was shown that flexible Lagrangians with non-empty boundary satisfy an existence and uniqueness h-principle, demonstrating that flexible Weinstein domains have many Lagrangians with Legendrian boundary. The slightly more general class of \textit{semi-flexible} Lagrangians was also introduced; these are constructed by considering a flexible Lagrangian in a flexible Weinstein domain and taking the boundary connected sum with an arbitrary Weinstein domain. However as shown in \cite{EGL}, semi-flexible Lagrangians have vanishing wrapped Floer homology and hence there is no general existence h-principle for semi-flexible Lagrangians in arbitrary Weinstein domains. For example, there is no semi-flexible representative of the cotangent fiber $T^*M_x \subset T^*M$; see Corollary 6.4 of \cite{EGL}. In this paper, we will give a general existence h-principle for Lagrangians with Legendrian boundary in an arbitrary Weinstein domain. Although the resulting Lagrangians cannot be flexible (or semi-flexible) in general, it is interesting to note that their construction uses flexible Lagrangians in a crucial way and in some sense generalizes the construction of semi-flexible Lagrangians.

We first recall the necessary differential-topological condition for a manifold to admit a Lagrangian embedding into a Weinstein domain. 
As in \cite{EGL}, a \textit{formal Lagrangian embedding} of $L$ into a Weinstein domain $W$ is a pair $(f,\Phi_t)$, where $f:(L, \partial L)\rightarrow (W,\partial W)$ is  a smooth embedding and $\Phi_t:TL\to TW$, $t\in[0,1]$,  is a homotopy of injective homomorphisms covering $f$  such that
$\Phi_0=df$ and $\Phi_1$ is a Lagrangian homomorphism, i.e. $\Phi_1(T_x L)\subset T_x W$ is a  Lagrangian  subspace for all $x\in L$.
Note that we do not impose any conditions on $\Phi_t$ restricted to $\partial L$ and that this definition makes sense even if $L$ does not have boundary. We say that two formal Lagrangians are formally Lagrangian isotopic if they are isotopic through formal Lagrangians. 
The following result says that this necessary condition for the existence of a Lagrangian embedding is in fact sufficient. 
\begin{theorem}\label{thm: EGL_h-principle}
Suppose that $L^n, n \ge 3$, has non-empty boundary and
admits a formal Lagrangian embedding into a Weinstein domain $W^{2n}$. Then $L^{n}$ admits a regular Lagrangian embedding into $W^{2n}$ in the same formal Lagrangian isotopy class.
\end{theorem} 
In particular, Theorem \ref{thm: EGL_h-principle} constructs many regular Lagrangians in arbitrary Weinstein domains. As explained before, these Lagrangians are in general not flexible (if $W$ is not flexible) nor semi-flexible. For example, if $L \cdot L' \ne 0$ for some closed exact Lagrangian $L' \subset W$, then $L$ has non-vanishing wrapped Floer homology and hence cannot be semi-flexible \cite{EGL}. Unlike flexible Lagrangians, these Lagrangians will not be unique in their formal class; see Theorem \ref{thm: lag_concordance} and Corollary \ref{cor: exotic_disks} below. 

We also note that the restriction $n \ge 3$ in Theorem \ref{thm: EGL_h-principle} cannot be removed. For example, there is a formal Lagrangian embedding of the punctured torus $T^2 \backslash D^2 \hookrightarrow B^4$ such that $S^1 = \partial(T^2\backslash D^2) \hookrightarrow S^3$ is the smooth unknot. However, it is known that any exact 2-dimensional Lagrangian whose Legendrian boundary is the smooth unknot must be a disk \cite{EFraser}. A 4-dimensional construction similar to the one in Theorem \ref{thm: EGL_h-principle} was considered by Yasui \cite{Yasui_disks_4}, who produced many Lagrangian disks in $B^4$. However, these disks necessarily have different formal classes, and even smooth isotopy classes.

\subsection{Decomposition of regular Lagrangians}
The proof of Theorem \ref{thm: EGL_h-principle} involves first using the existence h-principle from \cite{EGL} to realize the formal Lagrangian as a flexible Lagrangian in the flexible domain $W_{flex}^{2n}$  and then applying the following result from previous work \cite{Lazarev_critical_points}: any Weinstein domain $W^{2n}, n \ge 3,$ can be Weinstein homotoped to $W_{flex}^{2n}$ plus 
a smoothly trivial Weinstein cobordism $C^{2n}$. Here $W_{flex}^{2n}$ is the unique flexible Weinstein structure almost symplectomorphic to $W^{2n}$; a diffeomorphism $\phi: (W, \omega_W) \rightarrow (X, \omega_X)$ of two symplectic manifolds is an \textit{almost} symplectomorphism if $\phi^*\omega_X$ can be deformed through non-degenerate (but not necessarily closed) two-forms to $\omega_W$. 
Therefore the Lagrangians produced by Theorem \ref{thm: EGL_h-principle} can be decomposed as  flexible Lagrangians in $W_{flex}^{2n}$ that are extended trivially  in $C^{2n}$. The following result shows that in fact all regular Lagrangians with boundary are of this form. 
Here a Weinstein homotopy of $(W, W_0)$ for a Weinstein subdomain $W_0 \subset W$ will mean a Weinstein homotopy of the Weinstein cobordism $W\backslash W_0$ fixed on $\partial_-(W\backslash W_0)  = \partial_+ W_0$. If $\partial W_0 \cap \partial W \ne \emptyset$, then we consider $W\backslash W_0$ as a Weinstein cobordism with corners and require the homotopy to be fixed on these corners as well. For example, if $L \subset W^{2n}$ is a regular Lagrangian with Legendrian boundary, then $W\backslash T^*L$ is a Weinstein cobordism with corners $\partial ST^*L$, the boundary of the unit cotangent bundle $ST^*L$ of $L$. 
\begin{theorem}\label{thm: lag_concordance}
Let $L^n \subset W^{2n}, n \ge 3,$ be a regular Lagrangian with non-empty boundary. Then $(W^{2n}, L^n)$ is Weinstein homotopic to 
$(W^{2n}_{flex} \cup C^{2n}, L^n_{flex})$, where $C^{2n}$ is a  smoothly trivial Weinstein cobordism and $L^n_{flex} \subset W^{2n}_{flex}$ is a flexible Lagrangian. 
\end{theorem}
Here $L_{flex}^n \subset W^{2n}_{flex}$ is the unique flexible Lagrangian 
 \cite{EGL} that is formally isotopic to $L \subset W^{2n}$ under the almost symplectomorphism between $W_{flex}$ and $W$. 
So Theorem \ref{thm: lag_concordance} shows that all regular Lagrangians with boundary can be decomposed into a flexible Lagrangian and a smoothly trivial Lagrangian. However we note that it is an open problem whether all exact Lagrangians in Weinstein domains are regular. 
The condition that the Weinstein homotopy is done in the complement of $T^*L^n$ implies that there is an exact symplectomorphism $\phi: W^{2n} \rightarrow W_{flex}^{2n} \cup C^{2n}$ (or rather their completions) such that $\phi(L^n) = L_{flex}^n$. Of course, this does not imply that $L^n \subset W^{2n}$ is flexible since $L^n_{flex} \subset W^{2n}_{flex}$ ceases to be flexible once $C^{2n}$ is attached to $W_{flex}^{2n}$ to form $W^{2n}$. In general, flexible Lagrangians and Weinstein domains are defined only for $n \ge 3$. However a 4-dimensional version of Theorem \ref{thm: lag_concordance} was proven by Conway, Etnyre, and Tosun \cite{Conway_Etnyre_disks}  for regular Lagrangian disks $D^2$ in $B^{4}_{std}$. 
We also note the similarity between the decomposition $(W^{2n}_{flex} \cup C^{2n}, L^n_{flex})$ of arbitrary regular Lagrangians in Theorem \ref{thm: lag_concordance} and
the definition of semi-flexible Lagrangians \cite{EGL}, i.e. Lagrangians $L^n \subset W^{2n}$ such that $(W^{2n}, L^n)$ is Weinstein homotopic to $(W^{2n}_{flex} \natural W_0, L_{flex}^n)$ for $L_{flex}^n \subset W^{2n}_{flex}$ and some arbitrary Weinstein domain $W_0$. 
However the former decomposition is much more general than the latter because the attaching spheres of the Weinstein handles of $C^{2n}$ can link symplectically with the Legendrian boundary $\partial L_{flex}^n$.

A slight modification of Theorem \ref{thm: lag_concordance} implies the following decomposition result for disks in cotangent bundles of spheres. 
\begin{corollary}\label{cor: reg_disk}
Suppose $L^n, n \ge 3,$ is a regular Lagrangian disk in $T^*S^n$ with any Weinstein structure. Then $(T^*S^n, D^n)$ is Weinstein homotopic to 
$(T^*D^n \cup H^n_\Lambda, D^n)$.
\end{corollary}
Here $T^*D^n$ is equipped with the standard subcritical Weinstein structure, $\Lambda$ is some Legendrian in $ST^*D^n  = (S^{2n-1}, \xi_{std}) \backslash N(\partial D^n)$, and $D^n\subset T^*D^n \cup H^n_\Lambda$ corresponds to the zero-section of $T^*D^n$. Since $T^*D^n = B^{2n}_{std}$, Corollary \ref{cor: reg_disk} provides a presentation of $T^*S^n_{std}$, the standard cotangent bundle, using a single $n$-handle. There are many such Lagrangian disks $D^n \subset T^*S^n_{std}$ distinguished by wrapped Floer homology,  e.g. the graphical Lagrangian disks considered in \cite{Abouzaid_Seidel}. So  Corollary \ref{cor: reg_disk} seems to imply that there are many $\Lambda$ such that $T^*S^n_{std} = B^{2n}_{std} \cup H^n_{\Lambda}$.
However for the construction of $D^n$, it is possible that only the linking of $\Lambda$ with $\Lambda_{unknot}:= \partial D^n \subset \partial T^*D^n = \partial B^{2n}$ matters. For example, if $D^n \cdot S^n = k$, then $\Lambda_{unknot}:=\partial D^n \subset \partial T^*D^n$ and $\Lambda \subset \partial T^*D^n$ have linking number $k$. We do not know how much the Legendrian isotopy class of $\Lambda$ matters in this situation. 
\begin{question}
Suppose $D^n$ is a regular Lagrangian disk in the \textit{standard} cotangent bundle $T^*S^n_{std}$. Can we always take $\Lambda$ in Corollary \ref{cor: reg_disk} to be $\Lambda_{unknot}$?
\end{question}
A positive answer to this question would imply that the data of a regular disk $D^n \subset T^*S^n_{std}$ is just the data of two Legendrian unknots $\Lambda_{unknot,1} \coprod \Lambda_{unknot,2} \subset (S^{2n-1},\xi_{std})$ that are Legendrian linked with each other. This is equivalent to the data of a single Legendrian unknot in $(S^{2n-1}, \xi_{std}) \backslash N(\Lambda_{unknot}) = ST^*D^n$. A positive answer would also imply that any regular Lagrangian disk in $T^*S^n_{std}$ can be disjoined from some cotangent fiber.  Corollary \ref{cor: reg_disk} already shows that for any regular disk $D^n \subset T^*S^n_{std}$, there exists another regular disk $C^n$, namely the co-core of the handle $H^n_{\Lambda}$, which is disjoint from $D^n$ and generates the wrapped Fukaya category of $T^*S^n_{std}$; see \cite{chantraine_cocores_generate}.

We also point out that there is another natural decomposition of regular Lagrangian disks. As shown in \cite{EGL}, any regular Lagrangian disk can be presented as the co-core of an $n$-handle. 
However caving out this co-core, i.e. removing the $n$-handle, can often result in exotic Weinstein domains. For example, consider  an exotic Weinstein structure $\Sigma^{2n}$ on the ball $B^{2n}$, 
for example the structures constructed by McLean \cite{MM}, and let $D^n \subset \Sigma^{2n} \cup H^n_{\Lambda_{unknot}}$ be the co-core of the handle $H^n_{\Lambda_{unknot}}$. The result of caving out $D^n$ is precisely $\Sigma^{2n}$. In general, the caved out domain does not even have to be diffeomorphic to the ball. For example, if $D^n \cdot S^n = k$ in $T^*S^n$, then the caved out domain is a rational homology ball $B^{2n}_k$ with $H_{n-1}(B^{2n}_k; \mathbb{Z}) \cong \mathbb{Z}/k\mathbb{Z}$. Hence this co-core decomposition results in the data of a (possibly exotic) Weinstein structure $\Sigma^{2n}$ on $B^{2n}$ (or some other smooth domain) and a Legendrian in $\partial \Sigma^{2n}$. On the other hand, our decomposition in Corollary \ref{cor: reg_disk} just depends on the data of a Legendrian in the standard structure $(S^{2n-1}, \xi_{std}) \backslash N(\Lambda_{unknot}) = ST^*D^n$.

Of course closed Lagrangians cannot be decomposed as in Theorem \ref{thm: lag_concordance} since flexible Weinstein domains have no closed exact Lagrangians. However, a similar factorization result does hold in the closed case: all the topology (except the top handle) of a closed regular Lagrangian can be put in a flexible domain. For example, in \cite{EGL}, it was shown that for any closed smooth manifold $M^n, n \ge 3$, satisfying the appropriate formal conditions, there is a Weinstein structure $T^*S^n_{M}$ on $T^*S^n$ that contains $M$ as a regular Lagrangian. These examples are constructed by using the flexible Lagrangian existence h-principle \cite{EGL} to produce a flexible embedding $M^n\backslash D^n \subset B^{2n}_{std}$ and then attaching a handle to $\partial(M^n\backslash D^n ) \subset (S^{2n-1}, \xi_{std}) = \partial B^{2n}_{std}$, i.e. $(T^*S^n_M, M^n)$ is Weinstein homotopic to 
$(B^{2n}_{std} \cup H^n_{\partial (M\backslash D^n)}, (M^n\backslash D^n) \cup H^n_{\partial (M\backslash D^n)})$. So the interesting topology of $M^n$ is contained in the flexible domain $B^{2n}_{std}$. In fact, all regular Lagrangians in $T^*S^n$ can be constructed this way. More generally, we have the following result. 
\begin{theorem}\label{thm: all_reg_lag}
Suppose that $L^n$ is a closed regular Lagrangian in $W^{2n}, n \ge 3$. 
Then there exists a regular Lagrangian $L^n \backslash D^n$  in a flexible domain $V^{2n}_{flex}$ such that 
$(W^{2n}, L^n)$ is Weinstein homotopic to $(V^{2n}_{flex} \cup H^n_{\partial (L\backslash D^n)}, (L^n \backslash D^n) \cup H^n_{\partial (L\backslash D^n)})$. 
\end{theorem}
The point of Theorem \ref{thm: all_reg_lag} is that the Lagrangian $L^n \backslash D^n$ is in the \textit{flexible} structure $V_{flex}^{2n}$. If we allowed $V^{2n}$ to have an arbitrary Weinstein structure, then this result would follow immediately from the definition of regularity. By Theorem \ref{thm: lag_concordance}, the regular Lagrangian $L^n\backslash D^n \subset V_{flex}^{2n}$ can be further decomposed into the flexible Lagrangian $(L^n \backslash D^n)_{flex} \subset V_{flex}^{2n}$ 
plus a trivial extension in the Weinstein cobordism $V_{flex} \backslash i(V_{flex})$, where $i: V_{flex} \hookrightarrow V_{flex}$ is some Weinstein embedding.

In general the topology of $V_{flex}$ in Theorem \ref{thm: all_reg_lag} will depend on the formal class of $L^n \subset W^{2n}$. Except for this, we can essentially control the topology of $V_{flex}^{2n}$. For example, a slight modification of Theorem \ref{thm: all_reg_lag} shows that 
for any regular $M^n \backslash D^n \subset W^{2n}$, 
$(W^{2n} \cup H^n_{\partial M^n \backslash D^n}, 
M^n \backslash D^n \cup H^n_{\partial M^n \backslash D^n})$ is Weinstein homotopic to 
$(W^{2n}_{flex} \cup H^n_{\partial M^n \backslash D^n}, 
M^n \backslash D^n \cup H^n_{\partial M^n \backslash D^n})$. In particular, there is an exact symplectomorphism $\phi: W^{2n}\cup H^n_{\partial M^n \backslash D^n}  \rightarrow W_{flex}^{2n} \cup H^n_{\partial(M^n \backslash D^n)}$ such that $\phi(M^n) =  M^n$. However, this symplectomorphism does not induce a symplectomorphism between $W^{2n}$ and $W^{2n}_{flex}$ (as it shouldn't). 
This is because even though $\phi$ maps $M^n$ to $M^n$,  it does not map the co-core of $H^n_{\partial(M^n \backslash D^n)}$ in $W^{2n} \cup H^n_{\partial(M^n \backslash D^n)}$ to the co-core of $H^n_{\partial(M^n \backslash D^n)}$ in $W^{2n}_{flex} \cup H^n_{\partial(M^n \backslash D^n)}$, which would be needed to conclude that  $W^{2n}$ and $W^{2n}_{flex}$ are symplectomorphic. 
In particular, the two co-cores are two regular Lagrangian disks $D^n_1, D^n_2 \subset W^{2n} \cup H^n_{\partial(M^n \backslash D^n)}$ both intersecting $M^n$ in exactly one point such that
 caving out $D^n_1$ results in $W^{2n}$ but caving out $D^n_2$ results in $W^{2n}_{flex}$. In particular,  these disks can be smoothly isotopic but are not Lagrangian isotopic. 

We can rephrase the above discussion as follows. Let $\mathfrak{Lagrangian}(W^{2n}, L^n)$ denote all regular embeddings of a closed manifold $L^n$ into some fixed Weinstein domain $W^{2n}$. Then Theorem \ref{thm: all_reg_lag} shows that the following map induced by simultaneous handle-attachment is surjective: 
\begin{equation}\label{eqn: surjective}
\mathfrak{Lagrangian}(W_{flex}^{2n}, M^n\backslash D^n) \rightarrow \mathfrak{Lagrangian}(W^{2n} \cup H^n, (M^n \backslash D^n) \cup H^n_{\partial(M\backslash D^n)}).
\end{equation}
On the right-hand-side,  $M^n\backslash D^n$ is considered as a Lagrangian in $W^{2n}$, which is allowed to have an arbitrary Weinstein structure. In particular, the map is surjective even if we consider all Weinstein structures on the right-hand-side since the Weinstein structure on the left-hand-side is always flexible. A similar map just on the level of Weinstein domains is considered in \cite{Lazarev_critical_points} and is also shown to be surjective. Theorem \ref{thm: all_reg_lag} shows that surjectivity holds even when we consider Weinstein domains and Lagrangians simultaneously.  

Now we consider an  application of Theorem \ref{thm: all_reg_lag}.
\begin{corollary}\label{cor: reg_sphere_cotangent}
	Let $S^n$ be a regular Lagrangian sphere in $T^*S^n$, with any Weinstein structure. Then there exists a regular disk $D^n \subset B^{2n}_{std}$ such that $(T^*S^n, S^n)$ is Weinstein homotopic to $(B^{2n}_{std} \cup H^n_{\partial D^n}, D^n \cup H^n_{\partial D^n})$. 
\end{corollary}
For example, the zero-section $S^n_0 \subset T^*S^n_{std}$ can be obtained by attaching a handle to the boundary of the Lagrangian unknot $D^n_0 \subset B^{2n}_{std}$. It is unknown whether there are exotic Lagrangian spheres in $T^*S^n_{std}$ (or whether all Lagrangian spheres in $T^*S^n_{std}$ are regular). However Corollary \ref{cor: reg_sphere_cotangent} shows that all such hypothetical spheres come from Lagrangian disks in $B^{2n}_{std}$. In particular, the following map is surjective. 
\begin{equation}\label{eqn: surjective_cotangent}
\mathfrak{Lagrangian}(B^{2n}_{std}, D^n) \rightarrow \mathfrak{Lagrangian}(T^*S^n, S^n).
\end{equation}
Hence there are at least as many regular Lagrangians disks in $B^{2n}_{std}$ as regular Lagrangian spheres in $T^*S^n$, with some Weinstein structure. Since the latter set is infinite, so is the former. 
\begin{corollary}\label{cor: exotic_disks}
If $n \ge 4$, there are infinitely many different regular Lagrangian disks in the standard Weinstein ball $B^{2n}_{std}$.  
\end{corollary}
More precisely, these Lagrangian disks are not isotopic through exact Lagrangians with Legendrian boundary. In fact, there is no symplectomorphism of $B^{2n}_{std}$ taking the (completed) disks to each other. Hence these are non-flexible disks in a flexible Weinstein domain. The first such examples were found by Eliashberg, Ganatra, and the author \cite{EGL} using a different construction based on work of Murphy and Siegel \cite{MS}. The disks in \cite{EGL} were shown to be exotic because the complement Weinstein subdomains obtained by caving them out are also exotic \cite{MS}. Since it is not known how to cave out Lagrangians other than disks to produce Weinstein subdomains, it is not clear how to extend the method in \cite{EGL} to Lagrangians with more general topology. On the other hand, the construction in Corollary \ref{cor: exotic_disks} can be easily modified to create many exotic Lagrangians in $B^{2n}_{std}$ with non-trivial topology, i.e. there is no topological restriction on the Lagrangian $L^n$ in Theorem \ref{thm: all_reg_lag} (besides the necessary formal conditions).

Since flexible domains have vanishing symplectic homology, the Lagrangians in Corollary \ref{cor: exotic_disks} all have vanishing wrapped Floer homology. However they can be  distinguished by the Legendrian contact homology of their Legendrian boundaries (and so even their Legendrian boundaries are not isotopic). Even though the wrapped Floer homology of these disks vanishes, the Legendrian contact homology does not vanish since the Legendrian boundaries have exact Lagrangian fillings, namely the disks themselves.

Using the regular disks from Corollary \ref{cor: exotic_disks}, it is easy to construct exotic disks in the standard cotangent bundle $T^*S^n_{std}$ disjoint from the zero-section. Abouzaid-Seidel \cite{Abouzaid_Seidel} constructed exotic disks in $T^*S^n_{std}$ that intersect the zero-section $S^n_{std}$ many times; these were obtained by looking at the graphs of functions $f: D^n \rightarrow \mathbb{R}$ in $T^*D^n \subset T^*S^n$ and distinguished by their wrapped Floer homology with the zero-section. However, it was unclear whether there are any exotic disks that intersects $S^n_{std} \subset T^*S^n_{std}$ exactly once.  Any such disk is equivalent to $T^*S^n_x$ in the wrapped Fukaya category of $T^*S^n$ and hence cannot be distinguished via its wrapped Floer homology. 
Here we show that such exotic disks do in fact exist in abundance, distinguished again by the Legendrian contact homologies of their boundaries. 
\begin{corollary}\label{cor: disks_int_zero_section}
If $n \ge 4$, there are infinitely many different regular Lagrangian disks in the standard cotangent bundle $T^*S^n_{std}$ that intersect the zero-section $S^n_0 \subset T^*S^n_{std}$ exactly once.
\end{corollary}
Hence $S^n_0 \subset T^*S^n_{std}$ is a flexible Lagrangian (since its complement is a trivial Weinstein cobordism) but $S^n_0 \cup D^n$ is a non-flexible (singular) Lagrangian (or Weinstein subdomain).
To prove this corollary, we follow the same approach as for Corollary \ref{cor: exotic_disks} and show that a related handle-attachment map
\begin{equation}\label{eqn: surjective_plumbing}
\mathfrak{Lagrangian}(T^*S^n_{std}, S^n_0 \vee D^n) \rightarrow \mathfrak{Lagrangian}(T^*S^n\sharp_p T^*S^n, S^n\vee S^n)
\end{equation}
is surjective. The left-hand-side consists of singular regular Lagrangians $S^n_0 \vee D^n$ in $T^*S^n_{std}$, where $S^n_0$ is the zero-section as before; equivalently, we can think of these as regular Lagrangian disks which intersect the zero-section exactly once. On the right-hand-side, we consider arbitrary Weinstein structures on the plumbing $T^*S^n\sharp_p T^*S^n$ that contain the standard plumbing 
$(T^*S^n\sharp_p T^*S^n)_{std}$ as a subdomain, i.e. a singular Lagrangian $S^n \vee S^n$. The surjective of Equation \ref{eqn: surjective_plumbing} is a decomposition result for two plumbed spheres, similar to how the surjectivity of Equation \ref{eqn: surjective_cotangent} was a decomposition of a single sphere. 

We can summarize the approach in Corollaries \ref{cor: exotic_disks}, \ref{cor: disks_int_zero_section} as follows. Once we know that a certain class of structures have many distinct objects, i.e. display rigidity, we can conclude via a flexibility argument that related classes also have many distinct objects without using rigid techniques separately on this second class. In practice, the only method we currently have to show rigidity in the symplectic setting is through J-holomorphic curve invariants (Legendrian contact homology in the examples above). But the argument above shows that whatever method can be used to prove rigidity in the first class can also be used to prove rigidity in the second class. So although it is unclear if flexibility can ever create rigidity on its own, we see that flexibility can propagate rigidity once it is estabilished for some other reason. 

\subsection{Regular Lagrangian caps}

As noted before, it is an open question whether all exact Lagrangians in Weinstein domains are regular. However this is known to be false for Weinstein \textit{cobordisms} $W$ with non-empty negative end $\partial_-W$. Eliashberg and Murphy \cite{EM} showed that there is an existence h-principle for Lagrangians caps whose negative end in $\partial_-W$ is a loose Legendrian. Hence it is possible to construct exact non-regular Lagrangians in $W$ by applying the h-principle to a formal Lagrangian embedding whose complement cobordism has homology above the middle-dimension (and therefore cannot be a Weinstein cobordism for topological reasons).  
In fact, Murphy \cite{Murphy_symplectization} used the caps h-principle to construct non-regular \textit{closed}, exact Lagrangians in the symplectization of overtwisted contact manifolds. 

The proof of the Lagrangian caps h-principle involves Gromov's h-principle for Lagrangian immersions \cite{gromov_hprinciple} and a version of the Whitney trick (which relies on the looseness of the negative end) to remove the double-points of the immersion. These operations do not take the  ambient Weinstein structure into account and hence do not produce regular Lagrangians in general, as noted above. In particular, it was not known whether there is an existence h-principle for \textit{regular} Lagrangian caps with loose negative end. Here we show that such an h-principle does hold, assuming the necessary formal conditions. We say that a smooth cobordism $W^{2n}$ is an \textit{almost} Weinstein cobordism if $W^{2n}$ has an almost complex structure and admits a Morse function all of whose critical points have index at most $n$ (and is increasing, decreasing near the positive, negative boundaries of $W^{2n}$ respectively).
\begin{theorem}\label{thm: caps}
Let $W^{2n}, n \ge 3,$ be a Weinstein cobordism and $L^{n} \subset W^{2n}$ a formal Lagrangian cobordism such that $\partial_-L$ is formally isotopic to a loose Legendrian $\Lambda_- \subset \partial_-W$, $\partial_+L$ is formally isotopic to a Legendrian $\Lambda_+ \subset \partial_+W$, and $W^{2n} \backslash L^n$ is an almost Weinstein cobordism. Then there is a regular Lagrangian cobordism in $W^{2n}$ from $\Lambda_-$ to $\Lambda_+$ formally isotopic to $L$. The same holds if $\partial_+ L$ is empty.
\end{theorem}

Using the Lagrangian caps h-principle, Eliashberg and Murphy \cite{EM} also proved an h-principle for \textit{Liouville} embeddings of flexible Weinstein domains. More precisely, they showed that if a flexible Weinstein admits an  \textit{almost} symplectic embedding into some Liouville domain, then it admits a Liouville embedding into that Liouville domain, i.e. is a Liouville subdomain.
Since this construction was based on their Lagrangian caps h-principle, which produces non-regular Lagrangians, the resulting Weinstein domains are not necessarily \textit{Weinstein} subdomains and it was unknown when this is possible.
We will use Theorem \ref{thm: caps} to show that there is an h-principle for \textit{Weinstein} embeddings of flexible domains, again assuming the necessary formal conditions. 
\begin{corollary}\label{cor: flex_emb_hprinciple}
Suppose that $X^{2n}, n \ge 3,$ is a Weinstein domain and $W^{2n}_{flex}$ is a flexible domain that has an almost symplectic embedding $i: W^{2n}_{flex}\rightarrow X^{2n}$ such that $X^{2n} \backslash W^{2n}_{flex}$ is an almost Weinstein cobordism. Then $i$ is smoothly isotopic to a Weinstein embedding $j: W^{2n}_{flex} \rightarrow X^{2n}$, i.e. $j(W^{2n}_{flex})$ is a Weinstein subdomain of $X^{2n}$. 
\end{corollary}

\section{Proofs of Results}
We now give proofs of the results stated in the Introduction. We will need to use the following decomposition theorem from \cite{Lazarev_critical_points} and its variations. 
\begin{theorem}[\cite{Lazarev_critical_points}]\label{thm: flexible_subdomain}
If $W^{2n}, n \ge 3$, is a Weinstein domain, then $W^{2n}$ can be Weinstein homotoped to $W_{flex}^{2n} \cup C^{2n}$, where $C^{2n}$ is a smoothly trivial Weinstein cobordism with two Weinstein handles of index $n-1, n$.  
\end{theorem}
We briefly recall the proof of this result. Given two Legendrian spheres $\Lambda_0, \Lambda$ in $(Y, \xi) = \partial W_0$, we can  \textit{handle-slide} $\Lambda$ over $\Lambda_0$ and get a new Legendrian $h_{\Lambda_0}(\Lambda)$. Although $\Lambda_0 \coprod \Lambda$ and $\Lambda_0 \coprod h_{\Lambda_0}(\Lambda)$ may not be isotopic even as smooth links, the Weinstein structures $W_0 \cup H^n_{\Lambda_0} \cup H^n_{\Lambda}$ and $W_0 \cup H^n_{\Lambda_0} \cup H^n_{h_{\Lambda_0}(\Lambda)}$ are Weinstein homotopic.
A handle-slide depends on the choice of a local chart where $\Lambda_0, \Lambda$ look like parallel Legendrians and $h_{\Lambda_0}(\Lambda)$ is obtained by replacing $\Lambda$ with the cusp connected sum of $\Lambda_0, \Lambda$ in this local chart; see \cite{Casals_Murphy_front}. Theorem \ref{thm: flexible_subdomain} is essentially proven by showing that for any Legendrian link $\Lambda_0  \coprod \Lambda_1 \coprod \cdots \coprod \Lambda_k$, we can choose local charts so that the handle-slid link $h_{\Lambda_0}(\Lambda_1) \coprod \cdots \coprod h_{\Lambda_0}(\Lambda_k)$ is loose (but not loose in the complement of $\Lambda_0$). To be more precise, we actually need to do two handle-slides of opposition signs over $\Lambda_0$ to ensure that the cobordism $C^{2n}$ is smoothly trivial; see \cite{Lazarev_critical_points}.  
 
Now we use Theorem \ref{thm: flexible_subdomain} to prove Theorem \ref{thm: EGL_h-principle}.
\begin{proof}[Proof of Theorem \ref{thm: EGL_h-principle}]
We first describe the almost symplectomorphism between $W_{flex}$ and $W$ in more detail. By Theorem \ref{thm: flexible_subdomain}, $W_{flex} \subset W$ is a Weinstein subdomain and the Weinstein cobordism $C^{2n} = W\backslash W_{flex}$ consists of two handles, i.e. $C^{2n} = H^{n-1}_{\Lambda_0} \cup H^n_{\Lambda}$ where $\Lambda_0 \subset \partial W_{flex}$ and $\Lambda \subset \partial (W_{flex} \cup H^{n-1}_{\Lambda_0})$ are the isotropic attaching spheres. Since $C^{2n}$ is smoothly trivial, $\Lambda$ is smoothly isotopic in $ \partial (W_{flex} \cup H^{n-1}_{\Lambda_0})$  to a cancelling Legendrian that intersects the belt sphere of $H^{n-1}_{\Lambda_0}$ in a single point. We can assume that this smooth isotopy is supported in a neighborhood of some collection of  Whitney 2-disks with boundary on $\Lambda$ and the belt sphere of $H^{n-1}_{\Lambda_0}$.
Let $\phi_t$ be the extension of this smooth isotopy to an ambient diffeotopy of  $\partial (W_{flex} \cup H^{n-1}_{\Lambda_0})$, which is also supported in a neighborhood of these disks, and let $A \subset \partial W_{flex}$ be the subset where $\phi_t$ is the identity. 
Now we will show that there is a diffeomorphism $\psi$ between $W_{flex}$ and $W$ which is the identity on $A$. Here we view $A \subset \partial W_{flex} \backslash Op(\Lambda_0 \coprod \Lambda)$ as a subset of $\partial W$, where $Op(\Lambda_0 \coprod \Lambda)$ is a neighborhood 
$\Lambda_0, \Lambda$ in $\partial W_{flex}$ along which the handles are attached; note that $\Lambda$ intersects $\partial W_{flex}$ in a punctured Legendrian sphere whose boundary lies in $\Lambda_0$. To produce the diffeomorphism $\psi$,  we use the isotopy $\phi_t$ to homotope the gradient-like vector field in $W\backslash W_{flex}$ rel $\partial W_{flex}$ to a vector field with no zeroes and flow along that vector field. However this new vector field is no longer Liouville for the symplectic structure on $W\backslash W_{flex}$ and so $\psi$ will not be a symplectomorphism.  Since $\phi_t$ is the identity on $A$, this vector field is fixed over $A$ and so $\psi$ is the identity on $A$. We note that $\psi$ is an almost symplectomorphism since $W\backslash W_{flex}$ is smoothly trivial and any almost symplectic structure on a smoothly trivial cobordism is homotopic to the product structure. 

Let $F: L^n \rightarrow W^{2n}$ be the given formal Lagrangian embedding which we seek to realize by a genuine Lagrangian embedding. 
Then $\psi^{-1}\circ F$ is a formal Lagrangian in $W_{flex}^{2n}$. 
By the existence h-principle for Lagrangians in flexible Weinstein domains \cite{EGL}, $\psi^{-1}\circ F$ admits a flexible Lagrangian embedding $L_{flex}^n$ into $W_{flex}^{2n}$. We can Legendrian isotope the Legendrian boundary $\partial L_{flex}^n \subset \partial W_{flex}^{2n}$ so that it is disjoint from the isotropic attaching spheres $\Lambda_0, \Lambda$ of the two handles in $C^{2n}$. This is because a small neighborhood of $\partial L_{flex}$ is contactomorphic to $J^1(\partial L_{flex})$ and nearby Legendrians  are given by graphs of 1-jets of functions. Thom's jet transversality theorem shows that for any submanifold $\Sigma^k$ of $J^1(\partial L_{flex}^n)$ such that $k < n$, there exists a $C^0$-small function on $\partial L_{flex}^n$ whose 1-jet in $J^1(\partial L_{flex}^n)$  is disjoint from $\Sigma^k$; see Theorem 2.3.2 of  \cite{eliashberg_mishachev}. The isotropic attaching spheres $\Lambda_0^{n-2}, \Lambda^{n-1}$ of $C^{2n}$ have dimension less than $n$ and hence we can find such a Legendrian isotopy of $\partial L_{flex}$.
In particular, we can assume that $L_{flex} \subset W_{flex}$ is a flexible Lagrangian such that $\partial L_{flex}^n$ is disjoint from these attaching spheres. Furthermore, since $n \ge 3$, we can assume that $\partial L_{flex}^n$ is disjoint from the Whitney 2-disks inducing $\phi_t$ and hence contained in $A \subset \partial W_{flex}^{2n}$. 
Now we attach handles to $W_{flex}^{2n}$ along $\Lambda_0, \Lambda$ to form $W^{2n}$. Since $\partial L_{flex}^n$ is disjoint from these attaching spheres, it extends trivially to a Lagrangian with Legendrian boundary in $W^{2n}$ which we also call $L^n$. Since the cobordism $C^{2n}$ is Weinstein, $L^n$ is regular in $W^{2n}$.

Finally, we note that $L^n \subset W^{2n}$ is in the original Lagrangian formal class $F$. 
Since $\partial L_{flex}^n$ is contained in $A$ and the almost symplectomorphism $\psi$ between  $W_{flex}^{2n}$ and $W^{2n}$ is the identity on $A$, $\psi(L_{flex}^n)$  agrees with its trivial extension $L^n \subset W^{2n}$ described previously. Since $L_{flex}^n \subset W_{flex}^{2n}$ is in the formal class $\psi^{-1}\circ F$ by construction and the almost symplectomorphism $\psi$ preserves Lagrangian formal classes, $L^n = \psi(L_{flex}^n)$ will be in the desired formal class $F$.
\end{proof}
\begin{remark}
Even when $W = W_{flex}$, the almost symplectomorphism $\psi$ produced via the procedure above will not be a symplectomorphism. If $\psi$ were a symplectomorphism (of completions), then $L = \psi(L_{flex}) \subset \psi(W_{flex}) = W_{flex}$ would be a flexible Lagrangian. However as we will see later in Theorem \ref{thm: lag_concordance} and Corollary \ref{cor: exotic_disks}, all regular Lagrangians are of the form $\psi(L_{flex})$ but there are non-flexible Lagrangians even in flexible Weinstein domains. 	
\end{remark}

As noted before, Theorem \ref{thm: EGL_h-principle} does not hold in dimension 4 since there are Lagrangian formal classes not realized by any genuine Lagrangians. However an analogous construction in dimension four was considered by Yasui \cite{Yasui_disks_4}, who constructed many Lagrangians disks in $B^4_{std}$ by trivially extending the Lagrangian unknot $D^2 \subset T^*D^2 = B^4_{std}$ across a trivial Weinstein cobordism $(S^{3}, \xi_{std}) \times [0,1]$, presented as a Weinstein cobordism with two handles of index $1$ and $2$. These Lagrangian disks (and their Legendrian boundaries) are often in different formal classes, even different smooth isotopy classes; for example, there exist many smoothly slice knots in $S^3$ that are not isotopic to the unknot. Theorem \ref{thm: EGL_h-principle} is high-dimensional which gives us control over the formal class of the Lagrangian.

It is also crucial that the cobordism $C^{2n}$ is Weinstein. In this case, we can make the Legendrian boundary of $L_{flex}^n$ disjoint from the attaching Legendrians and extend $L_{flex}^n$ to a Lagrangian in $W^{2n}$.
So the key idea is that a Weinstein cobordism $C^{2n}$ modifies its negative contact boundary $\partial_- C^{2n}$ only in a small region, of dimension less than $n$. If we only knew that the cobordism had a Liouville structure, as shown earlier by Eliashberg and Murphy \cite{EM}, then we could not necessarily conclude that the Lagrangian extends since the Liouville cobordism could in principle modify the negative boundary $\partial_-C^{2n}$ in an arbitrarily large region. In particular, the following question is open. 
\begin{question}
Is there an existence h-principle for exact Lagrangians with Legendrian boundary in general Liouville domains?
\end{question}
Of course these Lagrangians will not be regular since the ambient domain is not Weinstein. A related question is which Liouville domains are \textit{non-degenerate} in the sense of Ganatra \cite{Ganatra_sym_cohm}. 

Now we prove Theorem \ref{thm: lag_concordance} that all regular Lagrangians with Legendrian boundary come from the construction in Theorem \ref{thm: EGL_h-principle}.

\begin{proof}[Proof of Theorem \ref{thm: lag_concordance} ]
Since $L^n$ is regular in $W^{2n}$, by definition
$(W^{2n}, L^n)$ is Weinstein homotopic to $(T^*L^n \cup X^{2n}, L^n)$ for some Weinstein cobordism $X^{2n}$. 
Then by Theorem \ref{thm: flexible_subdomain}, we can homotope $X^{2n}$ to $X_{flex}^{2n} \cup C^{2n}$, where $C^{2n}$ is smoothly trivial. The proof of this result involves Weinstein homotoping $X^{2n}$ to $X_{flex}^{2n} \cup C^{2n}$ relative to the closed contact manifold $\partial T^*L^n$. However, we can also do this Weinstein homotopy relative to $\partial T^*L^n \backslash \partial L^n$, i.e. view $X^{2n}$ as cobordism with corners $\partial ST^*L$ and require the homotopy to be fixed on the corners as well. This is because we can pick the Darboux balls and isotropic arcs in  \cite{Lazarev_critical_points} used to do the handle-slides away from $\partial L$. As a result, the attaching spheres of $X_{flex}^{2n}$ will be loose in the complement of $\partial L^n \subset \partial T^*L^n$ and hence $L^n \subset T^*L^n \cup X_{flex}^{2n} = W_{flex}^{2n}$ will be a flexible Lagrangian. We denote this Lagrangian by $L_{flex}^n \subset W_{flex}^{2n}$. When we attach $C^{2n}$ to $W_{flex}^{2n}$ to get $W^{2n}$, the Lagrangian $L_{flex}^n \subset W_{flex}^{2n}$ extends trivially to $L^n$ (as in the proof of Theorem \ref{thm: EGL_h-principle}). Hence $(W^{2n}, L^n)$ is Weinstein homotopic to $(W^{2n}_{flex} \cup C^{2n}, L_{flex}^n)$. 
\end{proof}

We can apply a modified version of Theorem \ref{thm: lag_concordance} to Lagrangian disks in $T^*S^n$ and prove Corollary \ref{cor: reg_disk}.  
\begin{proof}[Proof of Corollary \ref{cor: reg_disk}]
Since $D^n \subset T^*S^n$ is a regular Lagrangian, $(T^*S^n, D^n)$ is Weinstein homotopic to $(T^*D^n \cup C^{2n}, D^n)$. Here $T^*D^n$ has the standard Weinstein structure and $C^{2n}$ is a Weinstein cobordism with corners, which by the Whitney trick and fact that $n \ge 3$ has a smooth handle-body decomposition with a single handle of index $n$. By a slight variation on Theorem \ref{thm: flexible_subdomain} (see Theorem 1.1 of \cite{Lazarev_critical_points}), $C^{2n}$ can be Weinstein homotoped (relative to the corners $\partial ST^*D^n)$ to a Weinstein structure with a single Weinstein handle of index $n$. Hence $(T^*S^n, D^n)$ is Weinstein homotopic to $(T^*D^n \cup H^n_{\Lambda}, D^n)$ as desired. 
\end{proof}

Now we prove Theorem \ref{thm: all_reg_lag}, a version of Theorem \ref{thm: lag_concordance} for closed Lagrangians. 
\begin{proof}[Proof of Theorem \ref{thm: all_reg_lag}]
Since $M^n \subset W^{2n}$ is regular, by definition $(W^{2n}, M^n)$ is Weinstein homotopic to $(T^*M^n \cup C^{2n}, M^n)$ for some Weinstein cobordism $C^{2n}$. 
The $n$-handles of $C^{2n}$ are attached along a Legendrian link $\Lambda$ in $\partial(T^*M \cup C_{sub}^{2n})$, where $C_{sub}^{2n}$ is the subcritical part of $C^{2n}$. By attaching the single $n$-handle of $T^*M$ \textit{after} these $n$-handles of $C^{2n}$, we can consider $\Lambda$ as a Legendrian link in $\partial (T^*(M^n\backslash D^n) \cup C_{sub}^{2n})$.  
  
Now we handle-slide $\Lambda$ over $\partial (M^n\backslash D^n) \subset  \partial (T^*(M^n\backslash D^n) \cup C_{sub}^{2n})$ so that the resulting Legendrian link 
$h_{\partial(M^n\backslash D^n) }(\Lambda)$ is loose (but not loose in the complement of $\partial (M^n\backslash D^n) \subset \partial
(T^*(M^n\backslash D^n) \cup C_{sub})$). Since  $T^*(M^n \backslash D^n)$ is subcritical, $T^*(M^n \backslash D^n) \cup C_{sub}^{2n} \cup H^n_{h(\Lambda)}$ is a flexible Weinstein domain, which we denote by $V_{flex}^{2n}$. Furthermore $V^{2n}_{flex} \cup H^n_{\partial(M^n \backslash D^n)}=
T^*(M^n \backslash D^n) \cup C_{sub}^{2n} \cup  H^n_{h(\Lambda)} \cup H^n_{\partial(M^n\backslash D^n)}$ is homotopic to 
$T^*(M^n \backslash D^n) \cup C_{sub}^{2n} \cup H^n_{\Lambda} \cup H^n_{\partial(M^n \backslash D^n)}  = T^*M^n \cup C_{sub}^{2n} \cup H^n_{\Lambda} = T^*M^n  \cup C^{2n} = W^{2n}$, the original Weinstein structure. More precisely, this homotopy is just a Legendrian isotopy from $\Lambda$ to $h_{\partial(M^n\backslash D^n)}(\Lambda)$ in 
$\partial (T^*M^n \cup C_{sub}^{2n})$. This homotopy occurs above $T^*M^n$ and so $(V_{flex}^{2n} \cup H^n_{\partial(M^n\backslash D^n) }, M^n\backslash D^n \cup H^n_{\partial(M^n\backslash D^n)})$ is Weinstein homotopic to $(W^{2n}, M^n)$, as desired. 
\end{proof}

Applying Theorem \ref{thm: all_reg_lag} to Lagrangian spheres in cotangent bundles, we prove Corollary \ref{cor: reg_sphere_cotangent}. 
\begin{proof}[Proof of Corollary \ref{cor: reg_sphere_cotangent}]
By Theorem \ref{thm: all_reg_lag}, 
$(T^*S^n, S^n)$ is Weinstein homotopic to $(V_{flex}^{2n} \cup H^n_{\partial D^n}, D^n \cup H^n_{\partial D^n})$ for some regular Lagrangian disk $D^n$ in a flexible domain $V_{flex}$. 
Since $S^n \subset T^*S^n$ is regular, we have $[S^n] = \pm 1 \subset H_n(T^*S^n) \cong \mathbb{Z}$. The co-core $C^n$ of the handle $H^n_{\partial D^n}$ intersects $S^n$ in exactly one point and hence
$[C^n] = \pm 1 \in H_n(T^*S^n, \partial T^*S^n) \cong \mathbb{Z}$. Since $V_{flex}$ is obtained by caving out $C^n$ from $T^*S^n$, we see that $V_{flex}$ is a homology ball (with simply-connected boundary since $n \ge 3$). By the h-cobordism theorem, $V_{flex}$ must be diffeomorphic to the ball. Since $V_{flex}$ is flexible and the ball has a unique almost symplectic structure, $V_{flex}$ must be Weinstein homotopic to $B^{2n}_{std}$ by the h-principle for flexible Weinstein structures \cite{CE12}. 
\end{proof}

We can use Corollary \ref{cor: reg_sphere_cotangent} to produce many Lagrangian disks in the standard Weinstein ball and prove Corollary \ref{cor: exotic_disks}.
\begin{proof}[Proof of Corollary \ref{cor: exotic_disks}]
McLean constructed an exotic Weinstein ball $\Sigma^{2n}, n \ge 4,$ and showed that $\Sigma_k^{2n}:=\natural_{i=1}^k \Sigma^{2n}$ are non-symplectomorphic since they have different number of idempotents in their symplectic homology. Since symplectic homology is additive under boundary connected sum and $SH(T^*S^n_{std})$ has only one non-zero idempotent, $T^*S^n_{std} \natural \Sigma_k^{2n}$ are also non-symplectomorphic. Furthermore, each $T^*S^n_{std} \natural \Sigma_k^{2n}$ contains a regular Lagrangian sphere, i.e. the zero-section $S^n_0$ of $T^*S^n_{std}$. By Theorem \ref{thm: all_reg_lag}, there is a regular Lagrangian disk $D_k^{2n} \subset B^{2n}_{std}$ such that $(T^*S^n_{std} \natural \Sigma_k^{2n}, S^n_0)$ is Weinstein homotopic to $(B^{2n}_{std} \cup H^n_{\partial D^n_k}, D_k^{2n} \cup H^n_{\partial D^n_k})$. 
Since $T^*S^n_{std} \natural \Sigma_k^{2n}$ are not symplectomorphic for different $k$, neither are the Lagrangian disks $D^n_k \subset B^{2n}_{std}$. More explicitly, these disks can be distinguished by the Hochschild homology $LH^{Ho}(\partial D^n_k)$ of the Legendrian contact algebra of their Legendrian boundary, which equals $SH(T^*S^n_{std} \natural \Sigma^{2n}_k) $ by \cite{BEE12}. In particular, the Legendrian boundaries $\partial D^n_k$ of these disks are also non-isotopic. 
\end{proof}
\begin{remark}
Since the wrapped Floer homology of the Lagrangian disks vanishes, the \textit{linearized} Legendrian contact homology induced by the Lagrangian disk fillings is the singular homology of an $n$-disk. Hence this invariant is the same for all the Lagrangians and cannot be used to distinguish them. 
\end{remark}

Next we prove Corollary \ref{cor: disks_int_zero_section}:  there are many Lagrangian disks in the standard cotangent bundle intersecting the zero-section exactly once. 
\begin{proof}[Proof of Corollary \ref{cor: disks_int_zero_section}]
It suffices to prove that the handle-attachment map for plumbings in Equation \ref{eqn: surjective_plumbing} is surjective since there are infinitely many different Weinstein structures on $T^*S^n \sharp_p T^*S^n$ containing $(T^*S^n \sharp_p T^*S^n)_{std}$. For example 
$(T^*S^n \sharp_p T^*S^n)_{std}\natural \Sigma_k^{2n}$ is an infinite collection of such structures, where $\Sigma_k^{2n}$ are McLean's exotic Weinstein structures on the ball \cite{MM}.
 The  surjectivity of Equation \ref{eqn: surjective_plumbing} basically follows from a relative version of Theorem \ref{thm: all_reg_lag}, which was used to prove the surjectivity of Equations \ref{eqn: surjective}, \ref{eqn: surjective_cotangent}.
Namely, we can view
$(T^*S^n \sharp_p T^*S^n,  (T^*S^n \sharp_p T^*S^n)_{std})$ as 
$(T^*S^n_{std} \cup C^{2n}, S^n_0 \vee T^*S^n_x \cup L^n)$, where $C^{2n}$ is a Weinstein cobordism with  $\partial_-C^{2n} = \partial T^*S^n$  that admits a smooth Morse function with a single handle of index $n$ 
and $L^n \subset C^{2n}$ is a regular Lagrangian disk cap with $\partial_- L^n = \partial T^*S^n_x$. Then a  version of Theorem \ref{thm: all_reg_lag} for cobordisms implies that 
 $(C^{2n}, L^n)$  is Weinstein homotopic to 
$(W_{flex}^{2n}  \cup H^n_{\partial_+(S^{n-1} \times [0,1])}, S^{n-1} \times [0,1] \cup H^n_{\partial_+(S^{n-1} \times [0,1])})$. Here $W_{flex}^{2n}$ is a Weinstein cobordism with $\partial_- W_{flex}^{2n} = \partial T^*S^n$ and $S^{n-1} \times [0,1] \subset W_{flex}^{2n}$ is a regular Lagrangian cylinder with $\partial_+(S^{n-1} \times [0,1]) = \partial T^*S^n_x \subset \partial_-W_{flex}^{2n}$ and $\partial_-(S^{n-1} \times [0,1]) \subset \partial_+ W_{flex}^{2n}$.
 The flexible cobordism $W_{flex}^{2n}$ is smoothly trivial and hence is Weinstein homotopic to the trivial Weinstein structure $\partial T^*S^n \times [0,1]$. Then $(T^*S^n_{std} \cup W_{flex}^{2n}, S^n_0 \vee (T^*S^n_x \cup S^{n-1} \times [0,1]))$ 
is Weinstein homotopic to $(T^*S^n_{std}, S^n_0 \vee D^n)$ for some regular Lagrangian disk $D^n \subset T^*S^n_{std}$. Since $S^{n-1} \times [0,1] \subset W_{flex}^{2n}$, the disk $D^n \subset T^*S^n_{std}$ intersects the zero-section $S^n_0$ in precisely one point $T^*S^n_x \cap S^n_0 = \{x\}$. Furthermore, 
$(T^*S^n_{std} \cup H^n_{\partial D^n}, S^n_0 \vee D^n \cup H^n_{\partial D^n})$ is Weinstein homotopic to $(T^*S^n_{std} \cup C^{2n}, S^n_0 \vee T^*S^n_x \cup L^n)$ and hence to $(T^*S^n \sharp_p T^*S^n,  (T^*S^n \sharp_p T^*S^n)_{std})$.
\end{proof}

Now we prove  Theorem \ref{thm: caps}, a regular version of the Lagrangian caps h-principle due to Eliashberg and Murphy \cite{EM}. 
\begin{proof}[Proof of Theorem \ref{thm: caps}]
We will break down the proof into three cases.
First, we will prove the case when $W^{2n}$ is a flexible Weinstein cobordism and $\Lambda_-, \Lambda_+$ are both loose. 
Then we will prove the case when $W^{2n}, L^n$ are both smoothly trivial and $\Lambda_-, \Lambda_+$ are both loose. Finally, we will prove the case when $W^{2n}$ is smoothly trivial with the trivial product Weinstein structure, $L^n$ is smoothly trivial, $\Lambda_-$ is loose but $\Lambda_+$ is arbitrary. The general case follows by gluing the Lagrangians and Weinstein cobordisms produced in these three cases.

We first prove the case when $W^{2n}$ is a flexible Weinstein cobordism $W_{flex}$ and $\Lambda_-, \Lambda_+$ are both loose. 	
By the h-principle for flexible Lagrangians \cite{EGL}, there is a flexible Lagrangian cobordism $L_{flex} \subset W_{flex}^{2n}$ such that $\partial_-L_{flex} = \Lambda_-$ in $\partial_- W_{flex}^{2n}$ and $L$ is in the prescribed formal class. 
Recall that 
$\partial_-L_{flex} = \Lambda_-$ is loose by assumption. 
We will show that $\partial_+ L_{flex} \subset \partial_+ W_{flex}^{2n}$ is also loose. 
To see this, note that $L_{flex} \subset W_{flex}^{2n}$ is constructed in two steps: first we attach $T^*L$ to $\partial_-W_{flex}^{2n}$ along $\Lambda_-$ and then attach $W^{2n}_{flex} \backslash T^*L$. For the first step, suppose $S^{k-1} \subset \Lambda_-$ is an attaching sphere for a $k$-handle of $T^*L$. By the h-principle for loose Legendrians \cite{Murphy11}, we can assume that $\Lambda_-$ has a loose chart $U$ such that $\Lambda_- \cap U$ is a disk $D^{n-1} \subset \Lambda_-$.
Since $D^{n-1}$ is a disk, we can smoothly isotope $S^{k-1}$ in $\Lambda_-$ so that  $S^{k-1}, D^{n-1}$ are disjoint; note that this smooth isotopy is in fact an isotropic isotopy of $S^{k-1}$ in $\partial W_{flex}$  since $\Lambda_-$ is isotropic. Because $\Lambda_- \cap U = D^{n-1}$, we see that $S^{k-1}$ is disjoint from the loose chart $U$. So when we attach a handle along $S^{k-1}$, the loose chart persists and the resulting Legendrian will still be loose. 
Iterating this procedure, we see that $\partial_+ L \subset \partial_- W_{flex}^{2n} \cup \partial T^*L$ is also loose.
For the second step, the attaching spheres for $W^{2n}_{flex}\backslash T^*L$ are loose in the complement of $\partial_+ L \subset \partial_- W_{flex}^{2n} \cup \partial T^*L$
(this is what it means for $L_{flex }\subset W^{2n}_{flex}$  
to be a flexible Lagrangian) and so the loose chart of $\partial_+ L \subset \partial_- W_{flex}^{2n} \cup \partial T^*L$ again extends to a loose chart of $\partial_+ L_{flex} \subset \partial_+ W_{flex}^{2n}$.
Because $L_{flex}$ is in the correct formal class,
so is $\partial_+ L_{flex}$, i.e. formally Legendrian isotopic to $\Lambda_+$. Since $\partial_+L_{flex}, \Lambda_+$ are both loose, they are actually Legendrian isotopic by the loose Legendrian h-principle \cite{Murphy11}. 

Now we prove the second case when $W^{2n}, L^n$ are smoothly trivial and $\Lambda_-, \Lambda_+$ are both loose. By \cite{Lazarev_critical_points}, we can assume that $W^{2n}$ has a Weinstein presentation with two handles 
$H^{n-1}_{\Lambda_0}, H^n_{\Lambda}$ (although having precisely two handles will not really matter for the argument here). 
There is a Legendrian isotopy $\phi_t(\Lambda_-), t\in [0,1],$ of  $\Lambda_- \subset \partial_- W^{2n}$ so that $\phi_1(\Lambda_-)$ is loose in the complement of $\Lambda_0, \Lambda$ (of course $\phi_t(\Lambda_-)$ might cross $\Lambda_0, \Lambda$ during this isotopy). Let $L^n \subset W^{2n}$ be the concatenation of the graph of this isotopy in $\partial_-W^{2n} \times [0,1]$ with  the trivial extension of $\phi_1(\Lambda_-)$ in $W^{2n}$. Then $L^n$ is in the correct formal class and $\partial_+ L^n$ is loose by construction. Therefore $\partial_+ L^n$ is Legendrian isotopic to $\Lambda_+$. 

Next we prove the third case when $W^{2n}$  has the trivial product Weinstein structure  $(Y^{2n-1}, \xi) \times[0,1]$, $L^n$ is smoothly trivial, and $\Lambda_-$ is loose but $\Lambda_+$ is arbitrary. First we attach two symplectically cancelling handles $H^{n-1}_{\Lambda_0}, H^n_{\Lambda}$ to $W^{2n}$ such that $\Lambda_0, \Lambda \subset \partial_+ W^{2n}$ are contained in a Darboux ball. Now we handle-slide $\Lambda_+ \subset \partial_+ W^{2n}$ over $\Lambda$ twice (with opposite orientations) so that the resulting Legendrian $h^2(\Lambda_+) \subset \partial_+(W^{2n} \cup H^{n-1}_{\Lambda_0})$ is loose and intersects the belt sphere of $H^{n-1}_{\Lambda_0}$ algebraically zero times. Since we are working in the ball, which is simply-connected, and $n \ge 3$, we can use the Whitney trick to smoothly isotope $h^2(\Lambda_+)$ off the belt sphere. Since $h^2(\Lambda_+)$ is loose, we can actually Legendrian isotope $h^2(\Lambda_+)$ off this belt sphere and view $h^2(\Lambda_+)$ as a Legendrian in $ \partial_+ W^{2n}$. 
Let $L^n \subset W^{2n} \cup H^{n-1}_{\Lambda_0} \cup H^n_{\Lambda} = W^{2n}$ be the regular Lagrangian obtained by extending $h^2(\Lambda_+)$ trivially when $H^{n-1}, H^n_{\Lambda}$ are attached.  So $\partial_+ L^n \subset \partial_+ W^{2n}$ is Legendrian isotopic to $\Lambda_+$ and $\partial_-L^n  =h^2(\Lambda_+) = \subset \partial_- W^{2n}$ is loose.
Furthermore, $L^n$ is formally isotopic to a product Lagrangian. Since $\partial_+ L^n$ is Legendrian isotopic to $\Lambda_+$, 
$\partial_-L^n$ must be formally isotopic to $\Lambda_-$. Since $\partial_- L^n$ and $\Lambda_-$ are both loose, they are Legendrian isotopic by the h-principle for loose Legendrians \cite{Murphy11}. 
This finishes the proof of Theorem \ref{thm: caps} when $\partial_+L$ is non-empty. 

Finally, we prove the case when  $\partial_+ L$ is empty. We first realize $L$ as a flexible  Lagrangian in $W_{flex}$ with $\partial_- L = \Lambda_-$. We cannot directly apply the h-principle for flexible Lagrangians \cite{EGL} since $\partial_+ L = \emptyset$. Instead,  we first Weinstein homotope $W_{flex}^{2n}$ to $V_{flex}^{2n} \cup H^n_{\Lambda}$, for some loose Legendrian sphere $\Lambda$, such that $L\backslash D^n$ has a flexible Lagrangian embedding into $V_{flex}^{2n}$ with $\partial_+(L^n\backslash D^n) = \Lambda \subset \partial_+ V_{flex}^{2n}$. Then 
$L^n \backslash D^n \cup H^n_{\Lambda} \subset V_{flex}^{2n} \cup H^n_{\Lambda}$ is a flexible Lagrangian embedding 
$L^n \subset W_{flex}^{2n}$.  
Then we attach the Weinstein cobordism $W\backslash W_{flex}$ from  Theorem \ref{thm: flexible_subdomain} to $W_{flex}^{2n}$ and obtain the desired regular Lagrangian $L\subset W^{2n}$. 
\end{proof}

We conclude by proving Corollary \ref{cor: flex_emb_hprinciple}, an h-principle for Weinstein embeddings of flexible domains. 
\begin{proof}[Proof of Corollary \ref{cor: flex_emb_hprinciple}]
By Theorem \ref{thm: flexible_subdomain}, $X_{flex}$ is a Weinstein subdomain of $X$ such that $X\backslash X_{flex}$ is smoothly trivial. Also, we can realize the almost Weinstein cobordism $X\backslash W_{flex}$ by a flexible cobordism $C_{flex}$ by Eliashberg's existence h-principle \cite{CE12}. Then $W_{flex} \cup C_{flex}$ is a flexible domain that is almost symplectomorphic to $X_{flex}$. So by the h-principle for flexible domains \cite{CE12}, $W_{flex} \cup C_{flex}$ is Weinstein homotopic to $X_{flex}$.
 In particular, $W_{flex}$ is a Weinstein subdomain of $X_{flex}$ and hence a Weinstein subdomain of $X$. Since $X\backslash X_{flex}$ is smoothly trivial, this new embedding is smoothly isotopic to the original embedding.
\end{proof}
\begin{remark}
Alternatively, we can prove Corollary \ref{cor: flex_emb_hprinciple} 
by first constructing a Weinstein embedding of $W_{sub}$ into $X$ and then using Theorem \ref{thm: caps} to find a regular embedding of the cores of the $n$-handles of $W_{flex}$ into $X\backslash W_{sub}$. This second proof is more along the lines of Eliashberg and Murphy's original proof \cite{EM} of this result for Liouville embeddings, which directly uses their Lagrangian caps h-principle. 
\end{remark}

\bibliographystyle{abbrv}
\bibliography{sources}

\begin{thebibliography}{10}

\bibitem{Abouzaid_Seidel}
M.~Abouzaid and P.~Seidel.
\newblock An open string analogue of {V}iterbo functoriality.
\newblock {\em Geom. Topol.}, 14(2):627--718, 2010.

\bibitem{BEE12}
F.~Bourgeois, T.~Ekholm, and Y.~Eliashberg.
\newblock Effect of {L}egendrian surgery.
\newblock {\em Geom. Topol.}, 16(1):301--389, 2012.
\newblock With an appendix by Sheel Ganatra and Maksim Maydanskiy.

\bibitem{Casals_Murphy_front}
R.~Casals and E.~Murphy.
\newblock Legendrian fronts for affine varieties, 2016.
\newblock arXiv:1610.06977.

\bibitem{chantraine_cocores_generate}
B.~Chantraine, G.~D. Rizell, P.~Ghiggini, and R.~Golovko.
\newblock Geometric generation of the wrapped {F}ukaya category of {W}einstein
  manifolds and sectors, 2017.
\newblock arXiv:1712.09126.

\bibitem{CE12}
K.~Cieliebak and Y.~Eliashberg.
\newblock {\em From {S}tein to {W}einstein and back}, volume~59 of {\em
  American Mathematical Society Colloquium Publications}.
\newblock American Mathematical Society, Providence, RI, 2012.
\newblock Symplectic geometry of affine complex manifolds.

\bibitem{Conway_Etnyre_disks}
J.~Conway, J.~B. Etnyre, and B.~Tosun.
\newblock Symplectic fillings, contact surgeries, and {L}agrangian disks, 2017.

\bibitem{EFraser}
Y.~Eliashberg and M.~Fraser.
\newblock Topologically trivial {L}egendrian knots.
\newblock {\em J. Symplectic Geom.}, 7(2):77--127, 2009.

\bibitem{EGL}
Y.~Eliashberg, S.~Ganatra, and O.~Lazarev.
\newblock Flexible {L}agrangians, 2015.
\newblock arXiv:1510.01287.

\bibitem{eliashberg_mishachev}
Y.~Eliashberg and N.~M. Mishachev.
\newblock {\em Introduction to the h-principle}.
\newblock Number~48. American Mathematical Soc., 2002.

\bibitem{EM}
Y.~Eliashberg and E.~Murphy.
\newblock Lagrangian caps.
\newblock {\em Geom. Funct. Anal.}, 23(5):1483--1514, 2013.

\bibitem{Ganatra_sym_cohm}
S.~Ganatra.
\newblock Symplectic cohomology and duality for the wrapped {F}ukaya category,
  2013.
\newblock arXiv:1304.7312.

\bibitem{gromov_hprinciple}
M.~Gromov.
\newblock {\em Partial differential relations}, volume~9 of {\em Ergebnisse der
  Mathematik und ihrer Grenzgebiete (3) [Results in Mathematics and Related
  Areas (3)]}.
\newblock Springer-Verlag, Berlin, 1986.

\bibitem{Lazarev_critical_points}
O.~Lazarev.
\newblock Simplifying {W}einstein {M}orse functions, 2018.
\newblock arXiv:1808.03676.

\bibitem{MM}
M.~McLean.
\newblock Lefschetz fibrations and symplectic homology.
\newblock {\em Geom. Topol.}, 13(4):1877--1944, 2009.

\bibitem{Murphy11}
E.~Murphy.
\newblock Loose {L}egendrian embeddings in high dimensional contact manifolds,
  2012.
\newblock arXiv:1201.2245.

\bibitem{Murphy_symplectization}
E.~Murphy.
\newblock Closed exact {L}agrangians in the symplectization of contact
  manifolds, 2013.
\newblock arXiv:1304.6620.

\bibitem{MS}
E.~Murphy and K.~Siegel.
\newblock Subflexible symplectic manifolds, 2015.
\newblock arXiv:1510.01867.

\bibitem{Yasui_disks_4}
K.~Yasui.
\newblock Maximal {T}hurston-{B}ennequin number and reducible {L}egendrian
  surgery.
\newblock 2015.

\end{thebibliography}

\end{document}